\documentclass[a4paper]{amsart}

\usepackage{amssymb, enumitem}
\usepackage{hyperref, aliascnt}

\newcommand{\andSep}{\,\,\,\text{ and }\,\,\,}
\newcommand{\axiomO}[1]{(O#1)}
\newcommand{\NN}{{\mathbb{N}}}

\newcommand{\ca}{$C^*$-algebra}
\newcommand{\CuSgp}{$\mathrm{Cu}$-sem\-i\-group}

\DeclareMathOperator{\Cu}{Cu}

\def\today{\number\day\space\ifcase\month\or   January\or February\or
   March\or April\or May\or June\or   July\or August\or September\or
   October\or November\or December\fi\   \number\year}

\newtheorem{lma}{Lemma}[section]

\newaliascnt{thmCt}{lma}
\newtheorem{thm}[thmCt]{Theorem}
\aliascntresetthe{thmCt}

\newaliascnt{corCt}{lma}
\newtheorem{cor}[corCt]{Corollary}
\aliascntresetthe{corCt}

\newaliascnt{prpCt}{lma}
\newtheorem{prp}[prpCt]{Proposition}
\aliascntresetthe{prpCt}

\theoremstyle{definition}

\newaliascnt{dfnCt}{lma}

\aliascntresetthe{dfnCt}

\newaliascnt{rmkCt}{lma}
\newtheorem{rmk}[rmkCt]{Remark}
\aliascntresetthe{rmkCt}

\newcounter{theoremintro}

\newaliascnt{thmIntroCt}{theoremintro}
\newtheorem{thmIntro}[thmIntroCt]{Theorem}
\aliascntresetthe{thmIntroCt}

\numberwithin{equation}{section}

\title{The Global Glimm Property for C*-algebras of topological dimension zero}

\author{Ping Wong Ng}
\address{Ping Wong Ng,
Department of Mathematics, University of Louisiana at Lafayette, 
P. O. Box 43568, Lafayette, LA, 70504–3568, USA.}
\email{png@louisiana.edu}

\author{Hannes Thiel}
\address{Hannes~Thiel, 
Department of Mathematical Sciences, Chalmers University of Technology and University of
Gothenburg, Gothenburg SE-412 96, Sweden.}
\email{hannes.thiel@chalmers.se}
\urladdr{www.hannesthiel.org}

\author{Eduard Vilalta}
\address{Eduard~Vilalta, 
Department of Mathematical Sciences, Chalmers University of Technology and University of
Gothenburg, Gothenburg SE-412 96, Sweden.}
\email{vilalta@chalmers.se}
\urladdr{www.eduardvilalta.com}

\thanks{
HT and EV were partially supported by the Knut and Alice Wallenberg Foundation (KAW 2021.0140). 
HT was also partially supported by the Swedish Research Council (VR) project grant 2024-04200.
EV was also partially supported by MINECO (grant No. PID2023-147110NB-I00) and by the Comissionat per Universitats i Recerca de la Generalitat de Catalunya (grant No. 2021 SGR 01015)
}

\subjclass[2020]%
{Primary
46L05; 
Secondary
19K14, 
46L80, 
46L85. 
}
\keywords{$C^*$-algebras, Global Glimm Property, nowhere scattered, Cuntz semigroups}
\date{\today}

\begin{document}

\begin{abstract}
We show that a \ca{} with topological dimension zero has the Global Glimm Property (every hereditary subalgebra contains an almost full nilpotent element) if and only if it is nowhere scattered (no hereditary subalgebra admits a finite-dimensional representation).
This solves the Global Glimm Problem in this setting.

It follows that nowhere scattered \ca{s} with finite nuclear dimension and topological dimension zero are pure. 
\end{abstract}

\maketitle

\section{Introduction}

The Global Glimm Problem asks if two relevant regularity properties ---known as \emph{nowhere scatteredness} and the \emph{Global Glimm Property}--- agree. 
The problem can be traced back to the pioneering study of purely infinite \ca{s} by Kirchberg and R\o{}rdam \cite{KirRor02InfNonSimpleCalgAbsOInfty}, and plays a central role in the study of nonsimple algebras; 
see, for example, \cite{AntPerThiVil24arX:PureCAlgs,BlaKir04GlimmHalving,EllRor06Perturb,RobTik17NucDimNonSimple}.

Concretely, a \ca{} $A$ is said to be \emph{nowhere scattered} if no hereditary sub-\ca{} of~$A$ admits a finite-dimensional, irreducible representation \cite{ThiVil24NowhereScattered}, while $A$ has the \emph{Global Glimm Property} if every hereditary sub-\ca{} contains an almost full square-zero element (in symbols, for every $a\in A_+$ and $\varepsilon>0$ there exists $r\in\overline{aAa}$ such that $r^2=0$ and $(a-\varepsilon)_+$ is in the ideal generated by $r$).

The Global Glimm Property always implies nowhere scatteredness, and the Global Glimm Problem asks if the converse is also true. 
This is known to hold under the additional assumption of real rank zero \cite{EllRor06Perturb}, or stable rank one \cite{AntPerRobThi22CuntzSR1}, or Hausdorff finite-dimensional primitive ideal space \cite{BlaKir04GlimmHalving}, or finite decomposition rank (\cite[Theorem~2.3]{EnsVil25arX:ZstableTwGp}).

Recently, a general framework to approach this problem was given in \cite{ThiVil23Glimm} employing Cuntz semigroup techniques. 
Making use of the tools developed in \cite{ThiVil23Glimm}, we solve the Global Glimm Problem for a large class of \ca{s}: 

\begin{thmIntro}[\ref{prp:GGP}]
\label{prp:thmIntro}
Let $A$ be a \ca{} with topological dimension zero. 
Then~$A$ has the Global Glimm Property if and only if $A$ is nowhere scattered.
\end{thmIntro}

A \ca{} is said to have \emph{topological dimension zero} if its primitive ideal space has a basis consisting of compact, open sets.
This notion was introduced by Brown and Pedersen \cite[Remark~2.5(vi)]{BroPed09Limits} as a generalization of real rank zero. 
More generally, it is known that every \ca{} with real rank zero has the \emph{ideal property} (that is, its projections separate its closed ideals), and that every \ca{} with the ideal property has topological dimension zero.

\autoref{prp:thmIntro} both recovers and extends previous solutions of the Global Glimm Problem.
First, since real rank zero implies topological dimension zero, our result subsumes the real rank zero case established in \cite{EllRor06Perturb}.
Second, because Hausdorff zero-dimensional primitive ideal space implies topological dimension zero, we partially recover the Hausdorff finite-dimensional primitive ideal space case treated in \cite{BlaKir04GlimmHalving}.
Beyond these cases, \autoref{prp:thmIntro} solves the Global Glimm Problem for classes of \ca{s} not covered by previous results in the literature.
For example, Zhang \cite{Zha90RieszDecomp} showed that multiplier algebras of $\sigma$-unital, real rank zero \ca{s} have topological dimension zero, and thus fall within the scope of \autoref{prp:thmIntro}.
However, such multiplier algebras are not encompassed by earlier results, since they need not have real rank zero, Hausdorff primitive ideal space, or stable rank one (see \autoref{rmk:MultNotRR0}).

\medskip

As a corollary of \autoref{prp:thmIntro} we obtain that nowhere scattered \ca{s} with finite nuclear dimension and topological dimension zero are pure (\autoref{prp:Pure}). 
Further, using results by Zhang \cite{Zha90RieszDecomp} and the third-named author \cite{Vil23arX:NWSMultCAlg}, we prove that $\sigma$-unital, purely infinite \ca{s} of real rank zero have purely infinite multiplier algebras (\autoref{prp:MultPI}), thereby partially resolving a problem from \cite{KirRor02InfNonSimpleCalgAbsOInfty}.
We also recover the result of Elliott and Rouzbehani \cite{EllRou23WkPITopDimZero} that a weakly purely infinite \ca{} with topological dimension zero is purely infinite (\autoref{prp:PI}).

\section{The proofs} 

Our results heavily rely on Cuntz semigroup techniques. 
The reader is referred to \cite{GarPer23arX:ModernCu} for an introduction to the subject.

As shown in \cite[Corollary~4.3]{PasRor07PIRR0}, a separable \ca{} has topological dimension zero if and only if $A\otimes\mathcal{O}_2$ has real rank zero. 
This was used in \cite[Proposition~4.18]{ThiVil23Glimm} to prove that a separable \ca{} $A$ has topological dimension zero if and only if $\Cu(A)\otimes\{ 0,\infty\}$ is algebraic, where we  recall that a Cuntz semigroup is \emph{algebraic} if the set of elements $x$ satisfying $x\ll x$ is sup-dense in $\Cu(A)$. 
Here, we use the tensor product of abstract Cuntz semigroups developed in \cite{AntPerThi18TensorProdCu}, and we note that $\Cu (A\otimes\mathcal{O}_2)\cong \Cu (A)\otimes\{ 0,\infty\}$ by \cite[Corollary~7.2.15]{AntPerThi18TensorProdCu}.

In the next result we remove the separability assumption.

\begin{prp}
\label{prp:TopDim0iffAlg}
Let $A$ be a \ca{}. 
Then, $A$ has topological dimension zero if and only if $\Cu(A)\otimes\{0,\infty\}$ is algebraic.
\end{prp}
\begin{proof}
Assume first that $\Cu(A)\otimes\{0,\infty\}$ is algebraic. 
It follows from \cite[Proposition~4.18]{ThiVil23Glimm} that $A$ has topological dimension zero.
(We note that the assumption of separability in \cite[Proposition~4.18]{ThiVil23Glimm} is not needed for this implication.)

\medskip

Conversely, assume that $A$ has topological dimension zero. 
Without loss of generality, we may also assume that $A$ is stable.
We will show that for every $x',x \in \Cu(A)$ with $x' \ll x$ there exist $y',y \in \Cu(A)$ such that $x' \ll y' \ll y \ll x$ and $y \ll \infty y'$.
It then follows from \cite[Lemma~4.16]{ThiVil23Glimm} that $\Cu(A)\otimes\{0,\infty\}$ is algebraic.

So let $x',x \in \Cu(A)$ satisfy $x' \ll x$.
Choose $a \in A_+$ and $\varepsilon>0$ such that
\[
x' \ll [(a-\varepsilon)_+], \andSep
x = [a].
\]

By \cite[Lemma~7.12]{RobTik17NucDimNonSimple}, for every $\varepsilon'>0$ the set of elements $c\in A_+$ such that $(c-\varepsilon')_+$ is pseudocompact is dense in $A_+$, where an element $d \in A_+$ is said to be \emph{pseudocompact} if there exist $\delta >0$ and $n\in\NN$ such that $[d]\leq n [(d-\delta )_+]$ in $\Cu (A)$. 
Applying this, we obtain $c \in A_+$ such that
\[
\| a - c \| < \frac{\varepsilon}{2}
\]
and such that $(c-\tfrac{\varepsilon}{2})_+$ is pseudocompact.
Using that $\| a - c \| < \tfrac{\varepsilon}{2}$, it follows from standard Cuntz semigroup techniques that
\[
(a-\varepsilon)_+ \precsim (c-\tfrac{\varepsilon}{2})_+ \precsim a.
\]
Thus, the pseudocompact element $d:=(c-\tfrac{\varepsilon}{2})_+$ satisfies $x' \ll [d] \leq x$.
(We have shown that elements in $\Cu(A)$ with a pseudocompact representative are sup-dense.)

Pick $\delta >0$ and $n\in\NN$ such that $[d]\leq n [(d-\delta )_+]$.
Using that $x' \ll [(a-\varepsilon)_+] \leq [d]$, we may assume that $\delta$ is so small that $x' \ll [(d-\delta )_+]$.
Set $y := [(d-\tfrac{\delta}{2})_+]$ and $y' := [(d-\delta )_+]$.
Then 
\[
x' \ll y' \ll y \ll [d] = [(c-\tfrac{\varepsilon}{2})_+] \leq [a] = x.
\]
Further, we have
\[
y \ll [d] \leq n[(d-\delta )_+] = ny' \leq \infty y',
\]
which shows that $y'$ and $y$ have the desired properties.
\end{proof}

As defined in \cite[Definition~5.1]{RobRor13Divisibility}, a \CuSgp{} $S$ is \emph{weakly $(2,\omega)$-divisible} if, for every $x',x\in S$ satisfying $x'\ll x$, there exist $n\in\NN$ and $d_1,\ldots ,d_n\in S$ such that 
\[
2d_1, \ldots, 2d_n \leq x, \andSep 
x'\leq d_1+\ldots +d_n.
\]

Further, $S$ is said to be \emph{$(2,\omega)$-divisible} if for every $x',x \in S$ satisfying $x' \ll x$ there exist $n\in\NN$ and $d\in S$ such that
\[
2d \leq x, \andSep 
x'\leq nd.
\]

As shown in \cite[Theorem~8.9]{ThiVil24NowhereScattered}, a \ca{} is nowhere scattered if and only if its Cuntz semigroup is weakly $(2,\omega)$-divisible.
Further, by \cite[Theorem~3.6]{ThiVil23Glimm}, a \ca{} has the Global Glimm Property if and only if its Cuntz semigroup is $(2,\omega)$-divisible. 
In what follows, we use these characterizations to prove \autoref{prp:thmIntro}.

We state \autoref{prp:GGP-Cu} below using the language of \CuSgp{s}, where properties \axiomO{5}–\axiomO{8} are always satisfied in the Cuntz semigroup of a \ca{} (see \cite{AntPerThi18TensorProdCu,AntPerRobThi21Edwards,ThiVil24NowhereScattered}). 
We recall \axiomO{5} at the point where it is used in the proof. 
The remaining properties \axiomO{6}-\axiomO{8} are required only for the referenced results and are therefore not recalled in this note.

\begin{lma}
\label{prp:GGP-Cu}
Let $S$ be a weakly $(2,\omega)$-divisible \CuSgp{} with \axiomO{5}-\axiomO{8}, and assume that $S \otimes \{0,\infty\}$ is algebraic.
Then $S$ is $(2,\omega)$-divisible.
\end{lma}
\begin{proof}
To show that $S$ is $(2,\omega)$-divisible, let $x',x \in S$ with $x' \ll x$.
We need to find $y \in S$ such that $2y \leq x$ and $x' \ll \infty y$.

By \cite[Lemma~4.17]{ThiVil23Glimm}, $S$ is ideal-filtered in the sense of \cite[Deﬁnition~4.1]{ThiVil23Glimm}.
We can therefore apply \cite[Proposition~6.2]{ThiVil23Glimm} for $m=2$ to obtain $c,d \in S$ such that
\[
2c, 2d \leq x, \andSep
x' \ll \infty(c+d).
\]
Choose $c',d' \in S$ such that $c' \ll c$, $d'\ll d$, and $x' \ll \infty(c'+d')$. Since $S\otimes\{0,\infty\}$ is algebraic, it follows from \cite[Lemma~4.16]{ThiVil23Glimm} that elements which generate a compact ideal are dense in $S$.
Specifically, applied for $c' \ll c$, we find $e',e \in S$ such that
\[
c' \ll e' \ll e \ll c, \andSep
e \ll \infty e'.
\]

The `almost algebraic order' property \axiomO{5} means that for all $s',s,t',t,u \in S$ satisfying $s' \ll s$, and $t' \ll t$, and $s+t \leq u$ there exists $v \in S$ such that $t' \ll v$ and $s'+v \leq u \leq s+v$.
Applying \axiomO{5} for $s'=2e'$, $s=2e$, $t'=t=0$ and $u=x$, we obtain $f\in S$ such that
\[
2e' + f 
\leq x
\leq 2e + f.
\]
Then 
\[
2d \leq x \leq f + 2e \leq f + \infty e.
\]

Using \cite[Proposition~7.8]{ThiVil24NowhereScattered} to `lift' the relation $2d\leq f +\infty e$ from the quotient by the ideal generated by $e$, we obtain $g \in S$ such that
\[
2g \leq f, \andSep
d' \ll g + \infty e.
\]
Set $y := e'+g$.
Then
\[
2y 
\leq 2e'+2g
\leq 2e'+f
\leq x.
\]
Further, we have
\[
x' 
\ll \infty(c'+d')
\leq \infty(e'+g+e)
\leq \infty(e'+g)
= \infty y,
\]
as desired.
\end{proof}

\begin{thm}
\label{prp:GGP}
Let $A$ be a \ca{} with topological dimension zero.
Then $A$ has the Global Glimm Property if and only if $A$ is nowhere scattered.
\end{thm}
\begin{proof}
The forward implication holds in general.
For the converse, assume that $A$ is nowhere scattered.
Then $\Cu(A)$ is weakly $(2,\omega)$-divisible by \cite[Theorem~8.9]{ThiVil24NowhereScattered}.
Further, the Cuntz semigroup of every \ca{} satisfies \axiomO{5}-\axiomO{8}.
Since $A$ has topological dimension zero, $\Cu(A)\otimes\{0,\infty\}$ is algebraic by \autoref{prp:TopDim0iffAlg}.

We can therefore apply \autoref{prp:GGP-Cu} to deduce that $\Cu(A)$ is $(2,\omega)$-divisible.
Then, by \cite[Theorem~3.6]{ThiVil23Glimm}, $A$ has the Global Glimm Property.
\end{proof}

In the rest of the paper we give some applications of \autoref{prp:GGP}.

\begin{cor}
\label{prp:Pure}
Let $A$ be a nowhere scattered \ca{} with finite nuclear dimension and topological dimension zero.
Then $A$ is pure. 
\end{cor}
\begin{proof}
The \ca{} $A$ has the Global Glimm Property by \autoref{prp:GGP}. Thus, the result follows from \cite[Theorem~6.5]{AntPerThiVil24arX:PureCAlgs}.
\end{proof}

In \cite[Proposition~4.15]{KirRor02InfNonSimpleCalgAbsOInfty}, Kirchberg and R{\o}rdam show that a \ca{} is purely infinite if and only if it is weakly purely infinite and has the Global Glimm Property, and they asked if every weakly purely infinite \ca{} must in fact be purely infinite (\cite[Question~9.5]{KirRor02InfNonSimpleCalgAbsOInfty}).
This question was settled affirmatively by Elliott and Rouzbehani  \cite{EllRou23WkPITopDimZero} for \ca{s} with topological dimension zero, a class that includes both the simple and the real rank zero cases already treated in \cite{KirRor02InfNonSimpleCalgAbsOInfty}.
We recover their result.

\begin{cor}[{Elliott-Rouzbehani \cite{EllRou23WkPITopDimZero}}]
\label{prp:PI}
Let $A$ be a \ca{} with topological dimension zero.
Then $A$ is purely infinite if and only if $A$ is weakly purely infinite.
\end{cor}
\begin{proof}
The forward implication holds in general.
For the converse, assume that $A$ is weakly purely infinite.
Then $A$ is nowhere scattered (\cite[Example~3.3]{ThiVil24NowhereScattered}), and therefore has the Global Glimm Property by \autoref{prp:GGP}.
Now it follows from \cite[Proposition~4.15]{KirRor02InfNonSimpleCalgAbsOInfty} that $A$ is purely infinite.
\end{proof}

We now turn to applications of \autoref{prp:GGP} to multiplier algebras, which rest on the following result of Zhang.

\begin{thm}[Zhang]
\label{prp:Zhang}
Let $A$ be a $\sigma$-unital \ca{} of real rank zero.
Then its multiplier algebra $\mathcal{M}(A)$ has topological dimension zero.
\end{thm}
\begin{proof}
By \cite[Theorem~2.2]{Zha90RieszDecomp}, every closed ideal in $\mathcal{M}(A)$ is the closed linear span of its projections.
Consequently, projections in $\mathcal{M}(A)$ separate closed ideals, which means that $\mathcal{M}(A)$ has the ideal property.
This implies that $\mathcal{M}(A)$ has topological dimension zero;
see, for example, \cite[Theorem~2.8]{PasPhi17WeakIPTopDimZero}.
\end{proof}

\begin{cor}
Let $A$ be a $\sigma$-unital, nowhere scattered \ca{} of real rank zero.
Then its multiplier algebra $\mathcal{M}(A)$ has the Global Glimm Property.
\end{cor}
\begin{proof}
By \autoref{prp:Zhang}, $\mathcal{M}(A)$ has topological dimension zero. 
Further, we know from \cite[Theorem~5.12]{Vil23arX:NWSMultCAlg} that $\mathcal{M}(A)$ is nowhere scattered. 
Using \autoref{prp:GGP}, we get that $\mathcal{M}(A)$ satisfies the Global Glimm Property.
\end{proof}

In \cite[Proposition~4.11]{KirRor02InfNonSimpleCalgAbsOInfty}, Kirchberg and R{\o}rdam show that if $A$ is a $\sigma$-unital, purely infinite \ca{}, then its multiplier algebra $\mathcal{M}(A)$ is weakly purely infinite, and they raise the question if $\mathcal{M}(A)$ is purely infinite.
We resolve this problem in the case that $A$ has real rank zero.

\begin{cor}
\label{prp:MultPI}
Let $A$ be a $\sigma$-unital, purely infinite \ca{} of real rank zero,
Then $\mathcal{M}(A)$ is purely infinite.
\end{cor}
\begin{proof}
By \autoref{prp:Zhang}, $\mathcal{M}(A)$ has topological dimension zero. 
Further, $\mathcal{M}(A)$ is weakly purely infinite by \cite[Proposition~4.11]{KirRor02InfNonSimpleCalgAbsOInfty}.
Now it follows from \autoref{prp:PI} that $\mathcal{M}(A)$ is purely infinite.
\end{proof}

\begin{rmk}
\label{rmk:MultNotRR0}
Let $A$ be a $\sigma$-unital, purely infinite \ca{} of real rank zero. 
Then~$\mathcal{M}(A)$ is purely infinite and has topological dimension zero by \autoref{prp:Zhang} and \autoref{prp:MultPI}.
However, $\mathcal{M}(A)$ need not have real rank zero.
For example, if~$Q$ denotes the Calkin algebra, then the stabilization $Q\otimes\mathcal{K}$ is a $\sigma$-unital, purely infinite \ca{} of real rank zero, but $\mathcal{M}(Q\otimes\mathcal{K})$ has real rank one by \cite[Example~3.4]{Thi24arX:RRMult}.
\end{rmk}


\begin{thebibliography}{APTV24}

\bibitem[APRT21]{AntPerRobThi21Edwards}
\bgroup\scshape{}R.~Antoine\egroup{}, \bgroup\scshape{}F.~Perera\egroup{},
  \bgroup\scshape{}L.~Robert\egroup{}, and \bgroup\scshape{}H.~Thiel\egroup{},
  Edwards' condition for quasitraces on \ca{s},  \emph{Proc. Roy. Soc.
  Edinburgh Sect. A} \textbf{151} (2021), 525--547.

\bibitem[APRT22]{AntPerRobThi22CuntzSR1}
\bgroup\scshape{}R.~Antoine\egroup{}, \bgroup\scshape{}F.~Perera\egroup{},
  \bgroup\scshape{}L.~Robert\egroup{}, and \bgroup\scshape{}H.~Thiel\egroup{},
  \ca{s} of stable rank one and their {C}untz semigroups,  \emph{Duke Math. J.}
  \textbf{171} (2022), 33--99.

\bibitem[APT18]{AntPerThi18TensorProdCu}
\bgroup\scshape{}R.~Antoine\egroup{}, \bgroup\scshape{}F.~Perera\egroup{}, and
  \bgroup\scshape{}H.~Thiel\egroup{}, Tensor products and regularity properties
  of {C}untz semigroups,  \emph{Mem. Amer. Math. Soc.} \textbf{251} (2018),
  viii+191.

\bibitem[APTV24]{AntPerThiVil24arX:PureCAlgs}
\bgroup\scshape{}R.~Antoine\egroup{}, \bgroup\scshape{}F.~Perera\egroup{},
  \bgroup\scshape{}H.~Thiel\egroup{}, and \bgroup\scshape{}E.~Vilalta\egroup{},
  Pure \ca{s}, preprint (arXiv:2406.11052 [math.OA]), 2024.

\bibitem[BK04]{BlaKir04GlimmHalving}
\bgroup\scshape{}E.~Blanchard\egroup{} and
  \bgroup\scshape{}E.~Kirchberg\egroup{}, Global {G}limm halving for
  {$C^*$}-bundles,  \emph{J. Operator Theory} \textbf{52} (2004), 385--420.

\bibitem[BP09]{BroPed09Limits}
\bgroup\scshape{}L.~G. Brown\egroup{} and \bgroup\scshape{}G.~K.
  Pedersen\egroup{}, Limits and \ca{s} of low rank or dimension,  \emph{J.
  Operator Theory} \textbf{61} (2009), 381--417.

\bibitem[ER06]{EllRor06Perturb}
\bgroup\scshape{}G.~A. Elliott\egroup{} and
  \bgroup\scshape{}M.~R{\o}rdam\egroup{}, Perturbation of {H}ausdorff moment
  sequences, and an application to the theory of \ca{s} of real rank zero,  in
  \emph{Operator {A}lgebras: {T}he {A}bel {S}ymposium 2004}, \emph{Abel Symp.}
  \textbf{1}, Springer, Berlin, 2006, pp.~97--115.

\bibitem[ER23]{EllRou23WkPITopDimZero}
\bgroup\scshape{}G.~A. Elliott\egroup{} and
  \bgroup\scshape{}M.~Rouzbehani\egroup{}, Weakly purely infinite \ca{s} with
  topological dimension zero are purely infinite,  \emph{C. R. Math. Acad. Sci.
  Soc. R. Can.} \textbf{45} (2023), 87--91.

\bibitem[EV25]{EnsVil25arX:ZstableTwGp}
\bgroup\scshape{}U.~Enstad\egroup{} and \bgroup\scshape{}E.~Vilalta\egroup{},
  $\mathcal{Z}$-stability of twisted group \ca{s} of nilpotent groups, preprint
  (arXiv:2503.18088 [math.OA]), 2025.

\bibitem[GP24]{GarPer23arX:ModernCu}
\bgroup\scshape{}E.~Gardella\egroup{} and \bgroup\scshape{}F.~Perera\egroup{},
  The modern theory of {C}untz semigroups of \ca{s}, EMS Surv. Math. Sci. (to
  appear), DOI: 10.4171/EMSS/84, 2024.

\bibitem[KR02]{KirRor02InfNonSimpleCalgAbsOInfty}
\bgroup\scshape{}E.~Kirchberg\egroup{} and
  \bgroup\scshape{}M.~R{\o}rdam\egroup{}, Infinite non-simple \ca{s}: absorbing
  the {C}untz algebra {$\mathcal{O}_\infty$},  \emph{Adv. Math.} \textbf{167}
  (2002), 195--264.

\bibitem[PP17]{PasPhi17WeakIPTopDimZero}
\bgroup\scshape{}C.~Pasnicu\egroup{} and \bgroup\scshape{}N.~C.
  Phillips\egroup{}, The weak ideal property and topological dimension zero,
  \emph{Canad. J. Math.} \textbf{69} (2017), 1385--1421.

\bibitem[PR07]{PasRor07PIRR0}
\bgroup\scshape{}C.~Pasnicu\egroup{} and
  \bgroup\scshape{}M.~R{\o}rdam\egroup{}, Purely infinite \ca{s} of real rank
  zero,  \emph{J. Reine Angew. Math.} \textbf{613} (2007), 51--73.

\bibitem[RR13]{RobRor13Divisibility}
\bgroup\scshape{}L.~Robert\egroup{} and \bgroup\scshape{}M.~R{\o}rdam\egroup{},
  Divisibility properties for \ca{s},  \emph{Proc. Lond. Math. Soc. (3)}
  \textbf{106} (2013), 1330--1370.

\bibitem[RT17]{RobTik17NucDimNonSimple}
\bgroup\scshape{}L.~Robert\egroup{} and \bgroup\scshape{}A.~Tikuisis\egroup{},
  Nuclear dimension and {$\mathcal{Z}$}-stability of non-simple \ca{s},
  \emph{Trans. Amer. Math. Soc.} \textbf{369} (2017), 4631--4670.

\bibitem[Thi24]{Thi24arX:RRMult}
\bgroup\scshape{}H.~Thiel\egroup{}, Real rank of some multiplier algebras, Isr.
  J. Math. (to appear), preprint (arXiv:2402.01022), 2024.

\bibitem[TV23]{ThiVil23Glimm}
\bgroup\scshape{}H.~Thiel\egroup{} and \bgroup\scshape{}E.~Vilalta\egroup{},
  The {G}lobal {G}limm {P}roperty,  \emph{Trans. Amer. Math. Soc.} \textbf{376}
  (2023), 4713--4744.

\bibitem[TV24]{ThiVil24NowhereScattered}
\bgroup\scshape{}H.~Thiel\egroup{} and \bgroup\scshape{}E.~Vilalta\egroup{},
  Nowhere scattered \ca{s},  \emph{J. Noncommut. Geom.} \textbf{18} (2024),
  231--263.

\bibitem[Vil23]{Vil23arX:NWSMultCAlg}
\bgroup\scshape{}E.~Vilalta\egroup{}, Nowhere scattered multiplier algebras,
  Proc. Roy. Soc. Edinburgh Sect. A (to appear), DOI: 10.1017/prm.2023.123,
  2023.

\bibitem[Zha90]{Zha90RieszDecomp}
\bgroup\scshape{}S.~Zhang\egroup{}, A {R}iesz decomposition property and ideal
  structure of multiplier algebras,  \emph{J. Operator Theory} \textbf{24}
  (1990), 209--225.

\end{thebibliography}

\providecommand{\bysame}{\leavevmode\hbox to3em{\hrulefill}\thinspace}
\providecommand{\noopsort}[1]{}
\providecommand{\mr}[1]{\href{http://www.ams.org/mathscinet-getitem?mr=#1}{MR~#1}}
\providecommand{\zbl}[1]{\href{http://www.zentralblatt-math.org/zmath/en/search/?q=an:#1}{Zbl~#1}}
\providecommand{\jfm}[1]{\href{http://www.emis.de/cgi-bin/JFM-item?#1}{JFM~#1}}
\providecommand{\arxiv}[1]{\href{http://www.arxiv.org/abs/#1}{arXiv~#1}}
\providecommand{\doi}[1]{\url{http://dx.doi.org/#1}}
\providecommand{\MR}{\relax\ifhmode\unskip\space\fi MR }
\providecommand{\MRhref}[2]{%
  \href{http://www.ams.org/mathscinet-getitem?mr=#1}{#2}
}
\providecommand{\href}[2]{#2}

\end{document}